\documentclass[11pt]{article}
\usepackage{amsfonts, amsmath, amsthm, amssymb, latexsym, amscd, color}
\usepackage{tabularx,ragged2e,booktabs,caption}
\usepackage{rotating}
\usepackage{mathrsfs,calrsfs}
\usepackage{textcomp}
\usepackage{graphicx,tikz}
\newcommand{\comment}[1]{}
\newif\ifpdf
\ifpdf \else \fi \textwidth = 6.5 in \textheight = 9 in
\newtheorem{thm}{Theorem}[section]

\newtheorem{observation}[thm]{Observation}

\newtheorem{lemma}[thm]{Lemma}
\newtheorem{definition}[thm]{Definition}
\newtheorem{theorem}[thm]{Theorem}

\begin{document}

\title{Domination in $4$-regular Kn\"odel graphs}
\author{{\small D.A. Mojdeh$^{a}$, S.R. Musawi}$^{b}$ and {\small E. Nazari$^{b}$}\\
\\{\small $^{a}$Department of
Mathematics, University of Mazandaran,}\\ {\small Babolsar, Iran}\\
{\small Email: damojdeh@yahoo.com}\\
{\small $^{b}$Department of Mathematics, University of Tafresh,}\\ {\small Tafresh, Iran}}
\date{}
\maketitle

\begin{abstract}
A subset $D$ of vertices of a graph $G$ is a \textit{dominating
set} if for each $u\in V(G) \setminus D$, $u$ is adjacent to some
vertex $v\in D$. The \textit{domination number}, $\gamma(G)$ of
$G$, is the minimum cardinality of a dominating set of $G$. For an even
integer $n\ge 2$ and $1\le \Delta \le \lfloor\log_2n\rfloor$, a
\textit{Kn\"odel graph} $W_{\Delta,n}$ is a $\Delta$-regular
bipartite graph of even order $n$, with vertices $(i,j)$, for
$i=1,2$ and $0\le j\le n/2-1$, where for every $j$, $0\le j\le
n/2-1$, there is an edge between vertex $(1, j)$ and every vertex
$(2,(j+2^k-1)$ mod (n/2)), for $k=0,1,\cdots,\Delta-1$. In this
paper, we determine the domination number in $4$-regular
Kn\"odel graphs $W_{4,n}$.

\end{abstract}

\section{introduction}
For graph theory notation and terminology not given here, we refer
to \cite{hhs}. Let $G=(V,E)$ denote a simple graph of order
$n=|V(G)|$ and size $m=|E(G)|$. Two vertices $u,v\in V(G)$ are
\textit{adjacent} if $uv\in E(G)$. The \textit{open neighborhood}
of a vertex $u\in V(G)$ is denoted by $N(u)=\{v\in V(G) | uv\in
E(G)\}$ and for a vertex set $S\subseteq V(G)$,
$N(S)=\underset{u\in S}{\cup}N(u)$. The cardinality of $N(u)$ is
called the \textit{degree} of $u$ and is denoted by $\deg(u)$,
(or $\deg_G(u)$ to refer it to $G$). The \textit{closed neighborhood}
of a vertex $u\in V(G)$ is denoted by $N[u]=N(u)\cup\{u\}$ and for a vertex set $S\subseteq V(G)$,
$N[S]=\underset{u\in S}{\cup}N[u]$. The \textit{maximum degree}
and \textit{minimum degree} among all vertices in $G$ are denoted
by $\Delta(G)$ and $\delta(G)$, respectively. A graph $G$ is a
\textit{bipartite graph} if its vertex set can partition to two
disjoint sets $X$ and $Y$ such that each edge in $E(G)$ connects
a vertex in $X$ with a vertex in $Y$. A set $D\subseteq V(G)$ is
a \textit{dominating set} if for each $u\in V(G) \setminus D$,
$u$ is adjacent to some vertex $v\in D$. The \textit{domination
number}, $\gamma(G)$ of $G$, is the minimum cardinality of a
dominating set of $G$. The concept of domination theory is a
widely studied concept in graph theory and for a comprehensive
study see, for example \cite{hhs}.

An interesting family of graphs namely \textit{Kn\"odel graphs}
have been introduced about 1975 \cite{k}, and they have been
studied seriously by some authors since 2001. For an even integer
$n\ge 2$ and $1\le \Delta \le \lfloor\log_2n\rfloor$, a
\textit{Kn\"odel graph} $W_{\Delta,n}$ is a $\Delta$-regular
bipartite graph of even order $n$, with vertices $(i,j)$, for
$i=1,2$ and $0\le j\le n/2-1$, where for every $j$, $0\le j\le
n/2-1$, there is an edge between vertex $(1, j)$ and every vertex
$(2,(j+2^k-1)$ mod (n/2)), for $k=0,1,\cdots,\Delta-1$ (see
\cite{xxyf}). Kn\"odel graphs, $W_{\Delta,n}$, are one of the
three important families of graphs that they have good properties
in terms of broadcasting and gossiping, see for example
\cite{gh}. It is worth-noting that any Kn\"odel graph is a Cayley
graph and so it is a vertex-transitive graph (see \cite{fr}).

Xueliang et. al. \cite{xxyf} studied the domination number in
$3$-regular Kn\"odel graphs $W_{3,n}$. They obtained exact
domination number for $W_{3,n}$. Mojdeh et al. \cite{mmnj} determined the total domination number in
$3$-regular Kn\"odel graphs $W_{3,n}$. In this paper, we determine the
domination number in $4$-regular Kn\"odel graphs $W_{4,n}$. The following is useful. 
\begin{theorem}[\cite{b}, \cite{was}]\label{th1}
For any graph $G$ of order $n$ with maximum degree $\Delta(G)$, $\frac{n}{1+\Delta(G)}\le\gamma(G)\le n-\Delta(G)$.
\end{theorem}

We need also the following simple observation from number theory.

\begin{observation}\label{obs0}
If $a$, $b$, $c$, $d$ and $x$ are positive integers such that
$x^a-x^b=x^c-x^d\ne0$, then $a=c$ and $b=d$.
\end{observation}

\section{Properties in the Kn\"odel graphs}
In this section we review some properties in the Kn\"odel graphs that are proved in \cite{mmnj}. Mojdeh et al. considered a re-labeling on the vertices of a
Kn\"odel graph as follows: we label $(1,i)$ by $u_{i+1}$ for each
$i=0,1,...,n/2-1$, and $(2,j)$ by $v_{j+1}$ for
$j=0,1,...,n/2-1$. Let $U=\{u_1,u_2,\cdots, u_{\frac{n}{2}}\}$
and $V=\{v_1,v_2,\cdots, v_{\frac{n}{2}}\}$. From now on, the
vertex set of each Kn\"odel graph $W_{\Delta,n}$ is $U\cup V$
such that $U$ and $V$ are the two partite sets of the graph. If
$S$ is a set of vertices of $W_{\Delta,n}$, then clearly, $S\cap
U$ and $S\cap V$ partition $S$, $|S|=|S\cap U|+|S\cap V|$,
$N(S\cap U)\subseteq V$ and $N(S\cap V)\subseteq U$. Note that
two vertices $u_i$ and $v_j$ are adjacent if and only if $j\in
\{i+2^0-1,i+2^1-1,\cdots,i+2^{\Delta-1}-1\}$, where the addition
is taken in modulo $n/2$. For any subset $\{u_{i_1},u_{i_2},\cdots,u_{i_k}\}$ of $U$ with
$1\le i_1 <i_2<\cdots< i_k\le\frac{n}{2}$, it is corresponded a
sequence based on the differences of the indices of $u_j$,
$j=i_1,...,i_k$, as follows.

\begin{definition}\label{def1}
For any subset $A=\{u_{i_1},u_{i_2},\cdots,u_{i_k}\}$ of $U$ with
$1\le i_1 <i_2<\cdots< i_k\le\frac{n}{2}$ we define a sequence
$n_1, n_2, \cdots, n_k$, namely \textbf{cyclic-sequence}, where
$n_j=i_{j+1}-i_j$ for $1\le j\le k-1$ and
$n_k=\frac{n}{2}+i_1-i_k$. For two vertices $u_{i_j},
u_{i_{j'}}\in A$ we define \textbf{index-distance} of  $u_{i_j}$
and $u_{i_{j'}}$ by $id(u_{i_j},
u_{i_{j'}})=min\{|i_j-i_{j'}|,\frac{n}{2}-|i_j-i_{j'}|\}$.
\end{definition}

\begin{observation}\label{obs1}
Let $A=\{u_{i_1},u_{i_2},\cdots,u_{i_k}\}\subseteq U$ be a set
such that $1\le i_1 <i_2<\cdots< i_k\le\frac{n}{2}$ and let $n_1,
n_2, \cdots, n_k$ be the corresponding cyclic-sequence of $A$.
Then,\\
(1) $n_1+n_2+\cdots+n_k=\frac{n}{2}$.\\
(2) If $u_{i_j}, u_{i_{j'}}\in A$, then $id(u_{i_j}, u_{i_{j'}})$
equals to sum of some consecutive elements of the cyclic-sequence
of $A$ and $\frac{n}{2}-id(u_{i_j}, u_{i_{j'}})$ is sum of the
remaining elements of the cyclic-sequence. Furthermore,
$\{id(u_{i_j},u_{i_{j'}}),\frac{n}{2}-id(u_{i_j},
u_{i_{j'}})\}=\{|i_j-i_{j'}|,\frac{n}{2}-|i_j-i_{j'}|\}$.
\end{observation}

Let $\mathscr{M}_{\Delta}=\{ 2^a-2^b:0 \leq b<a< \Delta \}$ for $\Delta\geq 2$.

\begin{lemma}\label{lem0}
In the Kn\"odel graph $W_{\Delta,n}$ with vertex set $U\cup V$,
$N(u_i)\cap N(u_j)\ne \emptyset$ if and only if $id(u_i,u_j)\in
\mathscr{M}_{\Delta}$ or $\frac{n}{2}-id(u_i,u_j)\in
\mathscr{M}_{\Delta}$.
\end{lemma}

\section{4-regular Kn\"odel graphs}
In this section we determine the domination number in 4-regular Kn\"odel graphs $W_{4,n}$. Note that $n\geq 16$ by the definition. For this purpose, we prove the following lemmas namely Lemma \ref{lem1}, \ref{lem2}, \ref{lem3}, \ref{lem4} and \ref{lem5}.

\begin{lemma}\label{lem1}
For each even integer $n\ge 16$, we have  $\gamma (W_{4,n})=2\lfloor\frac{n}{10}\rfloor+\left \{\begin{array}{cc}0&n\equiv0 \emph{ (mod 10)}\\2&n\equiv 2,4\emph{ (mod 10)} \end{array}  \right. .$
\end{lemma}
\begin{proof} 
First assume that $n\equiv 0$ (mod 10). Let $n=10t$, where $t\ge 2$. By Theoerem \ref{th1}, $\gamma (W_{4,n})\ge \frac{n}{5}=2t$. On the other hand, we can see  that the set $D=\{u_1,u_6,\cdots,u_{5t-4}\}\cup\{v_5,v_{10},\cdots,v_{5t}\}$ is a dominating set with $2t$ elements, and we have $\gamma (W_{4,n})=2t=2\lfloor\frac{n}{10}\rfloor$, as desired. 

Next assume that $n\equiv 2$ (mod 10). Let $n=10t+2$, where $t\ge 2$.  By Theoerem \ref{th1}, we have $\gamma (W_{4,n})\ge \frac{n}{5}>2t$. Suppose that $\gamma(W_{4,n})=2t+1$. Let $D$ be a minimum dominating set of $W_{4,n}$. Then by the Pigeonhole Principal either $|D\cap U|\le t$ or  $|D\cap V|\le t$. Without loss of generality, assume that $|D\cap U|\le t$. Let $|D\cap U|= t-a$, where $a\ge 0$. Then $|D\cap V|= t+1+a$. Observe that $D\cap U$ dominates at most $4t-4a$ vertices of $V$ and therefore $D$ dominates at most $(4t-4a)+(t+1+a)=5t-3a+1$ vertices of $V$. Since $D$ dominates all vertices of $V$, we have $5t-3a+1\ge 5t+1$ and so $a=0$, $|D\cap U|=t$ and $|D\cap V|=t+1$. Let $D\cap U=\{u_{i_1},u_{i_2},\cdots,u_{i_t}\}$ and  $n_1, n_2, \cdots, n_t$ be the cyclic-sequence of  $D\cap U$. By Observation \ref{obs1}, we have $\overset{t}{\underset{k=1}{\sum}}n_k=5t+1$ and therefore there exist some  $k$ such that $n_k\in\mathscr{M}_4= \{1,2,3,4,6,7\}$. Then by Lemma \ref{lem0}, $|N(u_k)\cap N(u_{k+1})\ge1$.  Hence, $D\cap U$ dominates at most $4t-1$ vertices of $V$, that is, $D$ dominates at most $(4t-1)+(t+1)=5t$ vertices of $V$, a contradiction. Now we deduce that $\gamma(W_{4,n})\geq 2t+2$. On the other hand the set $D=\{u_1,u_6,\cdots,u_{5t+1}\}\cup\{v_5,v_{10},\cdots,v_{5t}\}\cup\{v_{5t+1}\}$ is a dominating set for $W_{4,n}$ with $2t+2$ elements. Consequently,  $\gamma(W_{4,n})=2t+2$.

It remains to assume that $n\equiv 4$ (mod 10). Let $n=10t+4$, where $t\ge 2$. By Theorem \ref{th1}, we have $\gamma (W_{4,n})\ge \frac{n}{5}>2t$. Suppose that $\gamma(W_{4,n})=2t+1$. Let $D$ be a minimum dominating set of $\gamma (W_{4,n})$. Then by the Pigeonhole Principal either $|D\cap U|\le t$ or $|D\cap V|\le t$. Without loss of generality, assume that $|D\cap U|=t-a$ and $a\ge 0$. Then $|D\cap V|= t+1+a$. Observe that $D\cap U$ dominate at most $4(t-a)$ elements of $V$ and therefore $D$ dominates at most $4(t-a)+(t+1+a)=5t-3a+1$ vertices of $V$. Since $D$ dominates  all vertices of $V$, we have $5t-3a+1\ge |V|=5t+2 $ and so $-3a\ge 1$, a contradiction. Thus $\gamma (W_{4,n})>2t+1$. On the other hand, the set $\{u_1,u_6,\cdots,u_{5t+1}\}\cup\{v_5,v_{10},\cdots,v_{5t}\}\cup\{v_3\}$ is a dominating set with $2t+2$ elements. Consequently, $\gamma (W_{4,n})=2t+2$.
\end{proof}

\begin{lemma}\label{lem2}
For each even integer $n\ge 46$ with $n\equiv6 \emph{ (mod 10)}$, we have $\gamma (W_{4,n})=2\lfloor\frac{n}{10}\rfloor+3$.
\end{lemma}

\begin{proof}
Let $n=10t+6$ and $t\ge 4$. By Theorem \ref{th1}, we have $\gamma (W_{4,n})\ge \frac{n}{5}>2t+1$. The set $D=\{u_1,u_6,\cdots,u_{5t+1}\}\cup\{v_5,v_{10},\cdots,v_{5t}\}\cup\{v_2,v_3\}$ is a dominating set with $2t+3$ elements. Thus, $2t+2\le \gamma (W_{4,n})\le 2t+3$. We show that $\gamma(W_{4,n})=2t+3$. Suppose to the contrary that $\gamma(W_{4,n})=2t+2$. Let $D$ be a minimum dominating set of $W_{4,n}$. Then by the Pigeonhole Principal either $|D\cap U|\le t+1$ or  $|D\cap V|\le t+1$. Without loss of genrality, assume that $|D\cap U|= t+1-a$, where $a\ge 0$. Then $|D\cap V|= t+1+a$. Note that $D\cap U$ dominates at most $4t+4-4a$ vertices and therefore $D$ dominates at most $(4t+4-4a)+(t+1+a)=5t-3a+5$ vertices of $V$. Since $D$ dominates all vertices of $V$, we have $5t-3a+5\ge 5t+3$ and so $a=0$ and $|D\cap U|=|D\cap V|=t+1$. Also we have $|V-D|=4t+2\le |N(D\cap U)|\le 4t+4$, and similarly, $|D-U|=4t+2\le |N(D\cap V)|\le 4t+4$.

Let $D\cap U=\{u_{i_1},u_{i_2},\cdots,u_{i_{t+1}}\}$ and $n_1, n_2, \cdots, n_{t+1}$ be the cyclic-sequence of $D\cap U$.  By Observation \ref{obs1}, we have $\overset{t+1}{\underset{k=1}{\sum}}n_k=5t+3$ and therefore there exist $k'$ such that $n_{k'}<5$. Then $n_{k'} \in \mathscr{M}_4$ and by Lemma \ref{lem0}, $|N(u_{i_{k'}})\cap N(u_{i_{k'+1}})|\ge1$.  Hence, $D\cap U$ dominates at most $4t+3$ vertices from $V$ and therefore $4t+2\le N(D\cap U)\le 4t+3$.

If $|N(D\cap U)|=4t+3$, then for each $k\ne k'$ we have $n_{k} \notin \mathscr{M}_4$. If there exists  $k''\ne k'$ such that $n_{k''}\ge 8$, then $5t+3=\overset{t+1}{\underset{k=1}{\Sigma}}n_k\ge 1+8+5(t-1)=5t+4$, a contradiction, and (by symmetry) we have $n_1=n_2=\cdots =n_t=5$, $n_{t+1}=3$ and $D\cap U=\{u_1,u_6,\cdots,u_{5t+1}\}$. Observe that $D\cap U$ doesn't dominate vertices $v_3,v_5,v_{10},\cdots,v_{5t}$ and so $\{v_3,v_5,v_{10},\cdots,v_{5t}\}\subseteq D$. Thus $D=\{u_1,u_6,\cdots,u_{5t+1},v_3,v_5,v_{10},\cdots,v_{5t}\}$. But the vertices $u_4,u_5,u_{5t-2},u_{5t+2}$ are not dominated by $D$, a contradiction.

Thus, $|N(D\cap U)|=4t+2$. Then there exists precisely two pairs of vertices in $D\cap U$ with index-distances belongs to $\mathscr{M}_4$. 
If there exists an integer $1\le i'\le t+1$ such that $n_{i'}+n_{i'+1}\in \mathscr{M}_4$, then $\min\{n_{i'},n_{i'+1}\}\le 3$. Then  $\min\{n_{i'},n_{i'+1}\}\in \mathscr{M}_4$ and $\max\{n_{i'},n_{i'+1}\}\notin \mathscr{M}_4$. Now we have $\max\{n_{i'},n_{i'+1}\}=5$ and $\min\{n_{i'},n_{i'+1}\}\in\{1,2\}$, and $n_i\notin \mathscr{M}_4$, for each $i\notin\{i',i'+1\}$. Now a simple calculation shows that the equality $5t+3=\overset{t+1}{\underset{k=1}{\sum}}n_k$ does not hold. (Note that if each $n_i$ is less than 8, then we have $\overset{t+1}{\underset{k=1}{\sum}}n_k\le 2+5+5(t-1)=5t+2$; otherwise we have $\overset{t+1}{\underset{k=1}{\sum}}n_k\ge 1+5+8+5(t-2)=5t+4$.) Thus there exist exactly two indices $j$ and $k$ such that $n_j,n_k\in \mathscr{M}_4$ and $\{n_1+n_2,n_2+n_3,\cdots,n_t+n_{t+1},n_{t+1}+n_1\}\cap \mathscr{M}_4=\emptyset$. By this hypothesis, the only possible cases for the cyclic-sequence of $D\cap U$ are those demonstrated in Table 1.
\begin{table}[ht] 
\captionsetup{type=table} 
\begin{center}
\begin{tabular}{c||c|c|c|c|c|c|c|c|c}
case&1&2&3&4&5&6&7&8&$\cdots$
\\ \hline \hline
$n_1$&8&4&3&8&4&4&4&4&$\cdots$
\\ \hline
$n_2$&1&1&2&2&4&5&5&5&$\cdots$
\\ \hline
$n_3$&4&8&8&3&5&4&5&5&$\cdots$
\\ \hline
$n_4$&5&5&5&5&5&5&4&5&$\cdots$
\\ \hline
$n_5$&5&5&5&5&5&5&5&4&$\cdots$
\\ \hline
$n_6$&5&5&5&5&5&5&5&5&$\cdots$
\\ \hline
$\cdots$ &$\cdots$&$\cdots$&$\cdots$&$\cdots$&$\cdots$&$\cdots$&$\cdots$&$\cdots$&$\cdots$
\\ \hline
$n_{t+1}$&5&5&5&5&5&5&5&5&$\cdots$
\end{tabular}
\caption{$n=10t+6$}\label{t1}
\end{center}
\end{table}
Note that the each column of Table 1, shows the cyclic-sequnce of $D\cap U$. We show that each case is impossible. For this purpose, we show that the cyclic-sequnce of $D\cap U$  posed in the column $i$, for $i\geq 1$ is impossible.

\textbf{$i=1$)}. If $n_1=8, n_2=1, n_3=4, n_4=\cdots =n_{t+1}=5$ and $D\cap U=\{u_1,u_9,u_{10}, u_{14}, u_{19}, \cdots, u_{5t-1}\}$, then $D\cap U$ does not dominate the vertices $v_5,v_6,v_7,v_{18},v_{23},v_{28},\dots,v_{5t+3}$. Thus we have $D\cap V=\{v_5,v_6,v_7,v_{18},v_{23},v_{28},\dots,v_{5t+3}\}$. But $D$ does not dominate three vertices $u_8,u_{12},u_{13}$, a contradiction.

\textbf{$i=2$)}. If $n_1=4, n_2=1, n_3=8, n_4=\cdots =n_{t+1}=5$ and $D\cap U=\{u_1,u_5,u_6, u_{14}, u_{19}, \cdots, u_{5t-1}\}$. then $D\cap U$ does not dominate the vertices $v_{10},v_{11},v_{16},v_{18},v_{23},v_{28},\dots,v_{5t+3}$. Thus we have $D\cap V=\{v_{10},v_{11},v_{16},v_{18},v_{23},v_{28},\dots,v_{5t+3}\}$. But $D$ does not dominate three vertices $u_2,u_{12},u_{5t+1}$, a contradiction.

\textbf{$i\in \{3,4\}$)}. As before, we obtain that $u_8\notin D$. But $N(u_8)=\{v_8,v_9,v_{11},v_{15}\}\subseteq N(D\cap U)$ and therefore $N(u_8)\cap D=\emptyset$. Hence, $N[u_8]\cap D=\emptyset$ and $D$ does not dominate $u_8$, a contradiction.

\textbf{$i\geq 5$)}. As before, we obtain that $u_3\notin D$. But $N(u_3)=\{v_3,v_4,v_6,v_{10}\}\subseteq N(D\cap U)$ and therefore $N(u_3)\cap D=\emptyset$. Hence, $N[u_3]\cap D=\emptyset$ and $D$ does not dominate $u_3$, a contradiction.

Consequently, $\gamma(W_{4,n})=2t+3$, as desired.
\end{proof}

Lemma \ref{lem2}, determine the domination number of $W_{4,n}$ when $n\equiv6$ (mod 10) and $n\geq 46$. The only values of $n$ for $n\equiv6$ (mod 10) are thus $16,26$ and $36$. We study these cases in the following lemma.

\begin{lemma}\label{lem3}
For $n\in \{16,26,36\}$, we have:
\begin{center}
\emph {\begin{tabular}{c|c|c|c}
n&16&26&36\\
\hline
$\gamma (W_{4,n})$&4&7&8
\end{tabular}
}
\end{center}
\end{lemma}

\begin{proof}
For $n=16$, by Theorem \ref{th1} we have $\gamma (W_{4,16})\ge \frac{16}{5}>3$. On the other hand, the set $D=\{u_1,u_2,v_6,v_7\}$ is a dominating set for $W_{4,16}$, and therefore $\gamma (W_{4,16})=4$.

For $n=26$, by Theorem \ref{th1} we have $\gamma (W_{4,26})\ge \frac{26}{5}>5$. On the other hand the set $D=\{u_1,u_4,u_9,u_{10},v_1,v_2,v_6\}$ is a dominating set for $W_{4,26}$  and therefore $6\le\gamma (W_{4,26})\le7$. We show that $\gamma (W_{4,26})=7$. Suppose to the contrary, that $\gamma (W_{4,26})=6$. Let $D$ be a minimum dominating set for $W_{4,26}$. Then by the Pigeonhole Principal either $|D\cap U|\le 3$ or $|D\cap V|\le 3$. If $|D\cap U|=3-a$, where $a\ge 0$, then $|D\cap V|= 3+a$. Now, the elements of $D\cap U$ dominate at most $4(3-a)$ elements of $V$ and $D$ dominates at most $4(3-a)+(3+a)=15-3a$ vertices of $V$. Thus $15-3a\ge |V|=13$ that implies $a=0$ and $|D\cap U|=|D\cap V|=3$. Let $|D\cap U|=\{u_1,u_i,u_j\}$, where $1<i<j\le 13$ and $n_1=i-1,n_2=j-i,n_3=13+1-j$. Since $n_1+n_2+n_3=13$, we have $\{n_1,n_2,n_3\}\cap \mathscr{M}_4\ne \emptyset$. 

If $\mathscr{M}_4$ includs at least three numbers of $n_1,n_2,n_3,n_1+n_2,n_2+n_3,n_3+n_1$, then by Lemma \ref{lem0}, $D\cap U$ dominates at most $4\times3-3=9$ vertices of $V$ and $|D\cap V|\ge 13-9=4$, a contradiction.

If $\mathscr{M}_4$ includs exactly one number of $n_1,n_2,n_3$, then we have, by symmetry, $n_1=3,n_2=5,n_3=5$ and $D\cap U=\{u_1,u_4,u_9\}$, $N(D\cap U)=\{v_1,v_2,v_3,v_4,v_5,v_7,v_8,v_9,v_{10},v_{11},v_{12}\}$. Then $\{v_6,v_{13}\}\subseteq D$. But $\{u_1,u_4,u_9,v_6,v_{13}\}$ does not dominate the vertices $\{u_2,u_7,u_8,u_{11}\}$ and we need at least 2 other vertices to dominate this four vertices, and hence $|D|\ge 7$, a contradiction.

Thus, we assume that $\mathscr{M}_4$ includs two numbers of $n_1,n_2,n_3$ and  $\{n_1+n_2,n_2+n_3,n_3+n_1\}\cap \mathscr{M}_4=\emptyset$. We thus have five possibilities for the cyclic-sequence of $D\cap U$ taht are demonstrated in Table \ref{t2}. Note that the each column of Table \ref{t2}, shows the cyclic-sequence of $D\cap U$. We show that each case is impossible. For this purpose, we show that the cyclic-sequence of $D\cap U$  posed in the column $i$, for $i\geq 1$ is impossible.

\begin{table}[ht] 
\begin{center}
\captionsetup{type=table} 

\begin{tabular}{c||c|c|c|c|c}
 $case$&1&2&3&4&5
\\ \hline \hline
 $n_1$&1&1&2&2&4
\\ \hline
$n_2$&4&8&3&8&4
\\ \hline 
$n_3$&8&4&8&3&5
\end{tabular}
\caption{$n=26$}\label{t2}
\end{center}
\end{table} 

\textbf{$i=1$)} If $n_1=1,n_2=4,n_3=8$, then $D\cap U=\{u_1,u_2,u_6\}$ and $N(D\cap U)=\{v_1,v_2,\cdots,v_9,v_{13}\}$. Thus $\{v_{10},v_{11},v_{12}\}\subseteq D$ and $D=\{u_1,u_2,u_6,v_{10},v_{11},v_{12}\}$. But $D$ does not dominate the vertex $u_{13}$, a contradiction.

 \textbf{$i=2$)} If $n_1=1,n_2=8,n_3=4$, then $D\cap U=\{u_1,u_2,u_{10}\}$ and $N(D\cap U)=\{v_1,v_2,\cdots,v_{11},v_{13}\}$. Thus $\{v_6,v_7,v_{12}\}\subseteq D$ and $D=\{u_1,u_2,u_{10},v_6,v_7,v_{12}\}$. But $D$ does not dominate the vertex $u_8$, a contradiction.

 \textbf{$i=3$)} If $n_1=2,n_2=3,n_3=8$, then $D\cap U=\{u_1,u_3,u_6\}$, $N(D\cap U)=\{v_1,v_2,\cdots,v_{10},v_{13}\}$.Thus $\{v_5,v_{11},v_{12}\}\subseteq D$ and $D=\{u_1,u_3,u_6,v_5,v_{11},v_{12}\}$. But $D$ does not dominate the vertices $u_7$ and $u_{13}$.
 
 \textbf{$i=4$)} If $n_1=2,n_2=8,n_3=3$, then $D\cap U=\{u_1,u_3,u_{11}\}$ and\\ $N(D\cap U)=\{v_1,v_2,\cdots,v_6,v_8,v_{10},v_{11},v_{12}\}$. Thus $\{v_7,v_9,v_{13}\}\subseteq D$ and $D=\{u_1,u_3,u_{11},v_7,v_9,v_{13}\}$. But $D$ does not dominate the vertex $u_5$, a contradiction.

 \textbf{$i=5$)} If $n_1=4,n_2=4,n_3=5$, then $D\cap U=\{u_1,u_5,u_9\}$ and $N(D\cap U)=\{v_1,v_2,\cdots,v_{10},v_{12}\}$. Thus $\{v_7,v_{11},v_{13}\}\subseteq D$ and $D=\{u_1,u_3,u_{11},v_7,v_9,v_{13}\}$. But $D$ does not dominate the vertices $u_2$ and $u_3$.
 
Consequently, $\gamma (W_{4,26})=7$.

We now consider the case $n=36$. By Theorem \ref{th1}, $\gamma (W_{4,36})\ge \frac{36}{5}>7$. On the other hand, the set $D=\{u_1,u_2,u_{10},u_{11},v_6,v_7,v_{15},v_{16}\}$ is a dominating set for the graph $W_{4,36}$ and therefore $\gamma (W_{4,36})=8$.
\end{proof}
We now consider the case $n\equiv 8$ (mod $10$). For $n=18,28$ and $38$ we have the following lemma.

\begin{lemma}\label{lem4}
For $n\in \{18,28,38\}$, we have:
\begin{center}
\emph {\begin{tabular}{c|c|c|c|c|c|c}
n&18&28&38\\
\hline
$\gamma (W_{4,n})$&4&7&10
\end{tabular}
}
\end{center}
\end{lemma}

\begin{proof}
For $n=18$, by Theorem \ref{th1} we have  $\gamma (W_{4,18})\ge \frac{18}{5}>3$. But the set $D=\{u_1,u_2,v_6,v_7\}$ is a dominating set for $W_{4,18}$ and therefore $\gamma (W_{4,18})=4$.

For $n=28$, by Theorem \ref{th1} we have $\gamma (W_{4,28})\ge \frac{28}{5}>5$  and the set $D=\{u_1,u_6,u_{11},u_{13},v_3,v_5,v_9\}$ is a dominating set for $W_{4,28}$  and therefore $6\le\gamma (W_{4,28})\le7$. Suppose by the contrary $\gamma (W_{4,28})=6$ and $D$ is a minimum dominating set for $\gamma (W_{4,28})$. Then by the Pigeonhole Principal either $|D\cap U|\le 3$ or $|D\cap V|\le 3$. If $|D\cap U|=3-a$ and $a\ge 0$, then $|D\cap V|= 3+a$. Now, the elements of $D\cap U$ dominate at most $4(3-a)$ elements of $V$ and $D$ dominates at most $4(3-a)+(3+a)=15-3a$ vertices of $V$. Thus $15-3a\ge |V|=14 $ therefore $a=0$ and we have $|D\cap U|=|D\cap V|=3$. Let $|D\cap U|=\{u_1,u_i,u_j\}$ and $1<i<j\le 14$ and $n_1=i-1,n_2=j-i,n_3=14+1-j$. Since $n_1+n_2+n_3=14$ we have $\{n_1,n_2,n_3\}\cap \mathscr{M}_4\ne \emptyset$. If $\mathscr{M}_4$ includs at least two numbers of $n_1,n_2,n_3,n_1+n_2,n_2+n_3,n_3+n_1$ then $D\cap U$ dominates at most $4\times3-2=10$ vertices of $V$ and $|D\cap V|\ge 14-10=4$, a contradiction.\\
The only remaining case is that $\mathscr{M}_4$ includes exactly one of the three number $n_1,n_2$ and $n_3$ and also $\{n_1+n_2,n_2+n_3,n_3+n_1\}\cap \mathscr{M}_4= \emptyset$. By symmetry we have $n_1=4,n_2=5,n_3=5$ and  $D\cap U=\{u_1,u_5,u_{10}\}$, $N(D\cap U)=\{v_1,v_2,v_3,v_4,v_5,v_6,v_8,v_{10},v_{11},v_{12},v_{13}\}$ thus $\{v_{7},v_{9},v_{14}\}\subseteq D$ and $D=\{u_1,u_5,u_{10},v_{7},v_{9},v_{14}\}$ but $D$ does not dominate the vertex $u_{3}$. That is a contradiction and therefore $\gamma (W_{4,28})=7$. For $n=38$, by Theorem \ref{th1} we have $\gamma (W_{4,38})\ge \frac{38}{5}>7$  and the set $D=\{u_1,u_6,u_{11},u_{16},u_{18},$ $v_3,v_5,v_{10},v_{13},v_{15}\}$ is a dominating set for $W_{4,38}$  and therefore $8\le\gamma (W_{4,38})\le10$. Let $\gamma (W_{4,38})<10$ and $D$ is a dominating set for $\gamma (W_{4,38})$ with $|D|=9$. Then by the Pigeonhole Principal either $|D\cap U|\le 4$ or $|D\cap V|\le4$. If $|D\cap U|=4-a$ and $a\ge 0$, then $|D\cap V|= 5+a$. Now, the elements of $D\cap U$ dominate at most $4(4-a)$ elements of $V$ and $D$ dominates at most $4(4-a)+(5+a)=21-3a$ vertices of $V$. Thus $21-3a\ge |V|=19 $ that results $a=0$ and we have $|D\cap U|=4$ and $|D\cap V|=5$. Let $|D\cap U|=\{u_1,u_i,u_j, u_k\}$ , where $1<i<j<k\le 19$ and $n_1, n_2, n_3, n_4$ be the cyclic-sequence of $D\cap U$. Since $n_1+n_2+n_3+n_4=19$ we have $\{n_1,n_2,n_3,n_4\}\cap \mathscr{M}_4\ne \emptyset$. If $\mathscr{M}_4$ includs at least three numbers of $n_1,n_2,n_3,n_4,n_1+n_2,n_2+n_3,n_3+n_4,n_4+n_1$ then $D\cap U$ dominates at most $4\times4-3=13$ vertices of $V$ and $|D\cap V|\ge 19-13=6$, a contradiction.\\
If we wish that $\mathscr{M}_4$ includs exactly one numbers of  $n_1,n_2,n_3,n_4$, we have three cases: 
\begin{table}[ht] 
\captionsetup{type=table} 
\begin{center}

\begin{tabular}{c||c|c|c}
 $case$&1&2&3
\\ \hline \hline
 $n_1$&4&1&8
\\ \hline
$n_2$&5&8&1
\\ \hline 
$n_3$&5&5&5
\\ \hline 
$n_4$&5&5&5
\end{tabular}
\caption{$n=38$ with one $n_i$ in $\mathscr{M}_{\Delta}$}\label{t3}
\end{center}
\end{table} 

\textbf{i=1)} If $n_1=4,n_2=5,n_3=5,n_4=5$ and $D\cap U=\{u_1,u_5,u_{10},u_{15}\}$ thus $\{v_7,v_9,v_{14},v_{19}\}\subseteq D$  but $\{u_1,u_5,u_{10},u_{15},v_7,v_9,v_{14},v_{19}\}$ does not dominate the vertices $u_3$ and $u_{17}$. For dominating $u_3$ and $u_{17}$, we need two vertices and therefore $|D|\ge 10$, a contradiction.

\textbf{i=2)} If $n_1=1,n_2=8,n_3=5,n_4=5$ and $D\cap U=\{u_1,u_2,u_{10},u_{15}\}$ thus $\{v_6,v_7,v_{12},v_{14},v_{19}\}\subseteq D$ and $D=\{u_1,u_2,u_{10},u_{15},v_6,v_7,v_{12},v_{14},v_{19}\}$ but $D$ does not dominate  the vertices $u_8$ and $u_{17}$.

\textbf{i=3)} If $n_1=8,n_2=1,n_3=5,n_4=5$ and $D\cap U=\{u_1,u_9,u_{10},u_{15}\}$ thus $\{v_5,v_6,v_7,v_{14},v_{19}\}\subseteq D$ and $D=\{u_1,u_9,u_{10},u_{15},v_5,v_6,v_7,v_{14},v_{19}\}$ but $D$ does not dominate  the vertex $u_8$.

Now we consider the cases that $\mathscr{M}_4$ includs exactly two numbers of the cyclic-sequence $n_1,n_2,n_3,n_4$ and $\{n_1+n_2,n_2+n_3,n_3+n_4,n_4+n_1\}\cap \mathscr{M}_4= \emptyset$. By symmetry we have ten cases:
\begin{table}[ht]
\captionsetup{type=table} 
\begin{center}
\begin{tabular}{c||c|c|c|c|c|c|c|c|c|c}
 $case$&1&2&3&4&5&6&7&8&9&10
\\ \hline \hline
 $n_1$&1&2&4&9&3&9&8&3&3&3
\\ \hline
$n_2$&9&8&1&1&2&2&3&6&5&5
\\ \hline 
$n_3$&1&1&9&4&9&3&5&5&6&5
\\ \hline 
$n_4$&8&8&5&5&5&5&3&5&5&6
\end{tabular}
\caption{$n=38$ with two $n_i$ in $\mathscr{M}_{\Delta}$}\label{t4}
\end{center}
\end{table}

\textbf{i=1)} If $n_1=1,n_2=9,n_3=1,n_4=8$ and $D\cap U=\{u_1,u_2,u_{11},u_{12}\}$ thus $\{v_6,v_7,v_{10},v_{16},v_{17}\}\subseteq D$ and $D=\{u_1,u_2,u_{11},u_{12},v_6,v_7,v_{10},v_{16},v_{17}\}$ but $D$ does not dominate  the vertex $u_8$. 
  
\textbf{i=2)} If $n_1=2,n_2=8,n_3=1,n_4=8$ and $D\cap U=\{u_1,u_3,u_{11},u_{12}\}$ thus $\{v_5,v_7,v_9,v_{16},v_{17}\}\subseteq D$ and  $D=\{u_1,u_3,u_{11},u_{12},v_5,v_7,v_9,v_{16},v_{17}\}$ but    $D$ does not dominate  the vertex $u_{18}$.   

\textbf{i=3)} If $n_1=4,n_2=1,n_3=9,n_4=5$ and $D\cap U=\{u_1,u_5,u_6,u_{15}\}$ thus  $\{v_{10},v_{11},v_{14},v_{17},v_{19}\}\subseteq D$ and   $D=\{u_1,u_5,u_6,u_{15},v_{10},v_{11},v_{14},v_{17},v_{19}\}$ but    $D$ does not dominate  the vertex $u_2$.  
   
\textbf{i=4)} If $n_1=9,n_2=1,n_3=4,n_4=5$ and $D\cap U=\{u_1,u_{10},u_{11},u_{15}\}$ thus  $\{v_{5},v_{6},v_{7},v_{9},v_{19}\}\subseteq D$ and   $D=\{u_1,u_{10},u_{11},u_{15},v_{5},v_{6},v_{7},v_{9},v_{19}\}$ but    $D$ does not dominate  the vertices $u_{13}$ and $u_{14}$. 

\textbf{i=5)} If $n_1=3,n_2=2,n_3=9,n_4=5$ and $D\cap U=\{u_1,u_4,u_6,u_{15}\}$ thus  $\{v_{10},v_{12},v_{14},v_{17},v_{19}\}\subseteq D$ and   $D=\{u_1,u_4,u_6,u_{15},v_{10},v_{12},v_{14},v_{17},v_{19}\}$ but    $D$ does not dominate  the vertices $u_{2}$ and $u_{8}$. 

\textbf{i=6)} If $n_1=9,n_2=2,n_3=3,n_4=5$ and $D\cap U=\{u_1,u_{10},u_{12},u_{15}\}$ thus  $\{v_{5},v_{6},v_{7},v_{9},v_{14}\}\subseteq D$ and   $D=\{u_1,u_{10},u_{12},u_{15},v_{5},v_{6},v_{7},v_{9},v_{14}\}$ but    $D$ does not dominate  the vertex $u_{16}$. 

\textbf{i=7)} If $n_1=8,n_2=3,n_3=5,n_4=3$ and $D\cap U=\{u_1,u_9,u_{12},u_{17}\}$ thus  $\{v_{3},v_{6},v_{7},v_{11},v_{14}\}\subseteq D$ and   $D=\{u_1,u_9,u_{12},u_{17},v_{3},v_{6},v_{7},v_{11},v_{14}\}$ but    $D$ does not dominate  the vertex $u_{16}$. 

\textbf{i=8)} If $n_1=3,n_2=6,n_3=5,n_4=5$ and $D\cap U=\{u_1,u_4,u_{10},u_{15}\}$ thus  $\{v_{6},v_{9},v_{12},v_{14},v_{19}\}\subseteq D$ and   $D=\{u_1,u_4,u_{10},u_{15},v_{6},v_{9},v_{12},v_{14},v_{19}\}$ but    $D$ does not dominate  the vertex $u_{17}$. 

\textbf{i=9)} If $n_1=3,n_2=5,n_3=6,n_4=5$ and $D\cap U=\{u_1,u_4,u_9,u_{15}\}$ thus  $\{v_{6},v_{13},v_{14},v_{17},v_{19}\}\subseteq D$ and   $D=\{u_1,u_4,u_9,u_{15},v_{6},v_{13},v_{14},v_{17},v_{19}\}$ but    $D$ does not dominate  the vertices $u_2$ and $u_8$. 

\textbf{i=10)} If $n_1=3,n_2=5,n_3=5,n_4=6$ and $D\cap U=\{u_1,u_4,u_{9},u_{14}\}$ thus  $\{v_{3},v_{6},v_{13},v_{18},v_{19}\}\subseteq D$ and   $D=\{u_1,u_4,u_{9},u_{14},v_{3},v_{6},v_{13},v_{18},v_{19}\}$ but    $D$ does not dominate  the vertices $u_{7}$ and $u_{8}$. 
      
   Hence, $W_{4,38}$ has not any dominating set with 9 vertices and  $W_{4,38}=10$ as desired.
\end{proof}

Now for $n\ge 48$ with $n\equiv8$ (mod 10) we determine the domination  number of $W_{4,n}$ as follows.

\begin{lemma}\label{lem5}
For each even integer $n\ge 48$, $n\equiv8$ \emph{(mod 10)}, we have  $\gamma (W_{4,n})=2\lfloor\frac{n}{10}\rfloor+4$.
\end{lemma}

\begin{proof} 
Let $n=10t+8$, where $t\ge 4$. By Theorem \ref{th1}, we have $\gamma (W_{4,n})\ge \frac{n}{5}>2t+1$. The set $D=\{u_1,u_6,\cdots,u_{5t+1}\}\cup\{v_5,v_{10},\cdots,v_{5t}\}\cup\{v_3,v_{5t-2},v_{5t+3}\}$ is a dominating set with $2t+4$ elements, and so, $2t+2\le \gamma (W_{4,n})\le 2t+4$. We show that $\gamma (W_{4,n})= 2t+4$.

First, assume that $\gamma(W_{4,n})=2t+2$. Let $D$ be a minimum dominating set of $W_{4,n}$. Then by the Pigeonhole Principal either $|D\cap U|\le t+1$ or $|D\cap V|\le t+1$. Without loss of generality, assume that $|D\cap U|= t+1-a$, where $a\ge 0$. Then $|D\cap V|= t+1+a$. Observe that $D\cap U$ dominates at most $4t+4-4a$ vertices of $V$, and therefore, $D$ dominates at most $(4t+4-4a)+(t+1+a)=5t-3a+5$ vertices of $V$. Since $D$ dominates all vertices in $V$, we have $5t-3a+5\ge 5t+4$ and so $a=0$. Then $|D\cap U|=|D\cap V|=t+1$. Also we have $|V-D|=4t+3\le |N(D\cap U)|\le 4t+4$ and $|U-D|=4t+3\le |N(D\cap V)|\le 4t+4$. Let $D\cap U=\{u_{i_1},u_{i_2},\cdots,u_{i_{t+1}}\}$ and $n_1, n_2, \cdots, n_{t+1}$.  By Observation \ref{obs1}, we have $\overset{t+1}{\underset{k=1}{\Sigma}}n_k=5t+4$ and therefore there exist $k'$ such that  $n_{k'} \in \mathscr{M}_4$. By Lemma \ref{lem0}, $D\cap U$ dominates at most $4t+3$ vertices from $V$. Then $|N(D\cap U)|=|N(D\cap V)|=4t+3$ and $k'$ is unique. If there exists  $1\le k''\le t+1$ such that $n_{k''}\ge 8$, then $5t+4=\overset{t+1}{\underset{k=1}{\Sigma}}n_k\ge n_{k'}+n_{k''}+5(t-1)\ge 1+8+5(t-1)=5t+4$ that implies $n_{k'}=1, n_{k''}=8$ and for each $k\notin\{k',k''\}$ we have $n_k=5$. Now in each arrangement of the cyclic-sequence of $D\cap U$, we have one adjacency between 1 and 5. Then we have two vertices in $D\cap U$ with index-distance equal to 6, a contradiction. Thus for $k\ne k'$ we have $n_k=5$ and $n_{k'}=4$. We have (by symmetry) $n_1=n_2=\cdots =n_t=5$ and $n_{t+1}=4$ and $D\cap U=\{u_1,u_6,\cdots,u_{5t+1}\}$. Now $D\cap U$ doesn't dominate the vertices $v_3,v_5,v_{10},\cdots,v_{5t}$ and so $\{v_3,v_5,v_{10},\cdots,v_{5t}\}\subseteq D$. Thus $D=\{u_1,u_6,\cdots,u_{5t+1},v_3,v_5,v_{10},\cdots,v_{5t}\}$, but $D$ does not dominate two vertices $u_{5t-2}$ and $u_{5t+3}$, a contradiction.
 
Now, assume that $\gamma(W_{4,n})=2t+3$. Let $D$ be a minimum dominating set of $W_{4,n}$. Then by the Pigeonhole Principal either $|D\cap U|\le t+1$ or $|D\cap V|\le t+1$. Without loss of generality, suppose $|D\cap U|= t+1-a$, where $a\ge 0$. Then $|D\cap V|= t+2+a$. Observe that $D\cap U$ dominates at most $4t+4-4a$ vertices of $V$ and therefore, $D$ dominates at most $(4t+4-4a)+(t+2+a)=5t-3a+6$ vertices of $V$. Since $D$ dominates all vertices in $V$, we have $5t-3a+6\ge 5t+4$ and $a=0$, $|D\cap U|=t+1$ and $|D\cap V|=t+2$. Also we have $4t+2\le |N(D\cap U)|\le 4t+4$. Since $4t+2\le |N(D\cap U)|\le 4t+4$, at most two elements of $n_1, n_2, \cdots, n_{t+1}$ can be in $\mathscr{M}_4$. If $x$ is the number of 5's in the cyclic-sequence of $D\cap U$, then by Observation \ref{obs1}, we have $\overset{t+1}{\underset{k=1}{\Sigma}}n_k=5t+4\ge 1+1+8(t-x-1)+5x$ and therefore $3x\ge 3t-10$ that implies $x\ge t-3$. Thus  $t-3$ elements of the cyclic-sequence are equal to 5. The sum of the remaining four values of the cyclic-sequence is 19, and at most two of them are in $\mathscr{M}_4$. In the last case of Lemma \ref{lem3}, for $n=38$, we identified all such cyclic-sequences and placed them in two tables, Table \ref{t3} and Table \ref{t4}. We now continue according to Table \ref{t3} and Table \ref{t4}.

In the case (i=1) in Table \ref{t3}, we have $n_1=4,n_2=\cdots=n_{t+1}=5$ and  $D\cap U=\{u_1,u_5,u_{10},\cdots,u_{5t}\}$. Thus $\{v_7,v_9,v_{14},\cdots,v_{5t+4}\}\subseteq D$. But $\{u_1,u_5,u_{10},\cdots,u_{5t},v_7,v_9,v_{14},\cdots,v_{5t+4}\}$ does not dominate the vertices $u_3$ and $u_{5t+2}$. For dominating $u_3$ and $u_{5t+2}$, we need two vertices and therefore $|D|\ge 2t+4$, a contradiction. In the cases (i$\in\{2,3\}$) in Table \ref{t3}, we have to add 5's to the end of the cyclic-sequence and construct the corresponding set $D$ with $2t+3$ elements. In both cases we obtain that $N[u_8]\cap D=\emptyset$. Then $D$ is not a dominating set, a contradiction.

In the case (i=1) in Table \ref{t4}, we can't add a 5 to the cyclic-sequence, since by adding a 5 to the cyclic-sequence we obtain two consecutive values of the cyclic-sequence which one is $5$ and the other is $1$ and their sum is $6$ which belongs to $\mathscr{M}_4$, a contradiction.

In the case (i=2) in Table \ref{t4}, we can't add a 5 to the cyclic-sequence, since by adding a 5 to the cyclic-sequence we obtain two consecutive values of the cyclic-sequence which one is $5$ and the other is $1$ or $2$, and their sum is $6$ or $7$, which belongs to $\mathscr{M}_4$, a contradiction.

In the cases (i$\in \{3,4,5,6\}$) in Table \ref{t4}, we have to add 5's to the end of the cyclic-sequence and construct the corresponding set $D$ with $2t+3$ elements. In (i=3), we obtain that $N[u_2]\cap D=\emptyset$, in (i=4), we obtain that $N[\{u_{13},u_{14}\}]\cap D=\emptyset$, in (i=5), we obtain that $N[\{u_2,u_8\}]\cap D=\emptyset$, and in (i=6), we obtain that $N[u_{16}]\cap D=\emptyset$. In all four cases, $D$ is not a dominating set, a contradiction.

In the case (i=7) in Table \ref{t4}, by adding 5's to the cyclic-sequence, we obtain some different new cyclic-sequences. We divide them into three categories.

c1) $n_1=8$, $n_2=3$ and $n_3=5$. In this category, the constructed set, $D$, does not dominate $u_{16}$, a contradiction.

c2) $n_1=8$, $n_2=5$, $n_3=3$ and $n_4=5$. In this category, the constructed set, $D$, does not dominate $u_{15}$ a contradiction.

c3) $n_1=8$, $n_2=n_3=5$ and if $n_i=n_j=3$, then $|i-j|\ge 2$. In this category, the constructed set, $D$, does not dominate $u_{5i+1}$, a contradiction. (Notice that this category does not appear for $n\le5$.) In the cases (i$\in\{8,9,10\}$) in Table \ref{t4}, by adding 5's to the cyclic-sequences, we obtain some different new cyclic-sequences. we divide them into two categories.

c1) $n_1=3$ and $n_{t+1}=5$. In this category, the constructed set, $D$, does not dominate $u_2$, a contradiction.

c2) $n_1=3$ and $n_{t+1}=6$. In this category, the constructed set, $D$, does not dominate $u_7$ and $u_8$. Then $D$ is not a dominating set, a contradiction.

Hence, $\gamma (W_{4,n})=2t+4=2\lfloor\frac{n}{10}\rfloor+4$
\end{proof}

Now a consequent of Lemmas \ref{lem1}, \ref{lem2}, \ref{lem3},\ref{lem4} and \ref{lem5} implies the following theorem which is the main result of this section.

\begin{theorem} For each integer $n\ge 16$, we have:
$$\gamma (W_{4,n})=2\lfloor\frac{n}{10}\rfloor+\left\{\begin{array}{cc}0&n\equiv0 \emph{ (mod 10)}\\2&n=16,18,36\,;\,n\equiv2,4 \emph{ (mod 10)} \\3&n=28\,;\,n\equiv6 \emph{ (mod 10)},n\ne16,36\\4&n\equiv8 \emph{ (mod 10)},n\ne18,28\end{array}  \right. .$$
\end{theorem}

\section{Conclusion} 

In this manuscript we study the domination number of $4$-regular Kn\"odel graphs. However there are some open related problems that  will be useful for studying.

\noindent \textbf{Problem 1} Domination number of $k$-regular Kn\"odel graphs for $k\ge 5$ obtain.\\
\textbf{Problem 2} Total domination number of $k$-regular Kn\"odel graphs for $k\ge 4$ obtain.\\
\textbf{Problem 3} Connected domination number of $k$-regular Kn\"odel graphs for $k\ge 3$ obtain.\\
\textbf{Problem 3} May be studied independent domination number of $k$-regular Kn\"odel graphs for $k\ge 3$ obtain.

\end{document}